\documentclass[a4paper,11pt]{amsart}
\title{S-duality for surfaces with $A_n$-type singularities}
\date{}
\author{Yukinobu Toda}

\usepackage{amscd}
\usepackage{amsmath}
\usepackage{amssymb}
\usepackage{amsthm}
\usepackage{float}
\usepackage[dvips]{graphicx}

\usepackage[all,ps,dvips]{xy}

\usepackage{array}
\usepackage{amscd}
\usepackage[all]{xy}

\DeclareFontFamily{U}{rsfs}{%
\skewchar\font127}
\DeclareFontShape{U}{rsfs}{m}{n}{%
<-6>rsfs5<6-8.5>rsfs7<8.5->rsfs10}{}
\DeclareSymbolFont{rsfs}{U}{rsfs}{m}{n}
\DeclareSymbolFontAlphabet
{\mathrsfs}{rsfs}
\DeclareRobustCommand*\rsfs{%
\@fontswitch\relax\mathrsfs}

\theoremstyle{plain}
\newtheorem{thm}{Theorem}[section]
\newtheorem{prop}[thm]{Proposition}
\newtheorem{lem}[thm]{Lemma}

\newtheorem{rmk}[thm]{Remark}
\newtheorem{cor}[thm]{Corollary}

\newtheorem{prop-defi}[thm]{Proposition-Definition}
\newtheorem{thm-defi}[thm]{Theorem-Definition}
\newtheorem{lem-defi}[thm]{Lemma-Definition}

\newtheorem{exam}[thm]{Example}

\newdimen\argwidth
\def\db[#1\db]{
 \setbox0=\hbox{$#1$}\argwidth=\wd0
 \setbox0=\hbox{$\left[\box0\right]$}
  \advance\argwidth by -\wd0
 \left[\kern.3\argwidth\box0 \kern.3\argwidth\right]}

\newcommand{\cC}{\mathcal{C}}

\newcommand{\lL}{\mathcal{L}}

\newcommand{\oO}{\mathcal{O}}

\newcommand{\Supp}{\mathop{\rm Supp}\nolimits}
\newcommand{\Hom}{\mathop{\rm Hom}\nolimits}

\newcommand{\dR}{\mathbf{R}}

\newcommand{\Hilb}{\mathop{\rm Hilb}\nolimits}

\newcommand{\ch}{\mathop{\rm ch}\nolimits}

\newcommand{\td}{\mathop{\rm td}\nolimits}

\newcommand{\Coh}{\mathop{\rm Coh}\nolimits}

\newcommand{\cneq}{\mathrel{\raise.095ex\hbox{:}\mkern-4.2mu=}}
\newcommand{\eqcn}{\mathrel{=\mkern-4.5mu\raise.095ex\hbox{:}}}

\newcommand{\Ex}{\mathop{\rm Ex}\nolimits}

\newcommand{\iPPer}{\mathop{\rm ^{-1}Per}\nolimits}

\newcommand{\pB}{\mathop{^{{p}}\mathcal{B}}\nolimits}

\newcommand{\iF}{\mathop{^{{-1}}\mathcal{F}}\nolimits}

\newcommand{\iT}{\mathop{^{{-1}}\mathcal{T}}\nolimits}

\newcommand{\Quot}{\mathop{\rm Quot}\nolimits}

\newcommand{\Ker}{\mathop{\rm Ker}\nolimits}

\newcommand{\length}{\mathop{\rm length}\nolimits}

\begin{document}
\maketitle

\begin{abstract}
We show that the generating series of 
Euler characteristics of Hilbert schemes of 
points on any algebraic 
surface with  
at worst $A_n$-type singularities is
described by the theta series determined by 
integer valued positive definite 
quadratic forms and the Dedekind eta 
function. In particular it is a
Fourier development of a meromorphic modular form
with possibly half integer weight. 
The key ingredient is to apply the flop transformation 
formula of Donaldson-Thomas type invariants
counting two dimensional torsion sheaves on 3-folds
proved in the author's previous paper. 
\end{abstract}

\section{Introduction}
\subsection{Background}
For an algebraic variety \footnote{In this paper, all 
the varieties are defined over $\mathbb{C}$.}$X$, the Hilbert scheme of 
$m$-points $\Hilb^m(X)$ is defined to be the moduli 
space of zero dimensional 
subschemes $Z \subset X$ such that 
the length of $\oO_Z$ equals to $m$. 
Its topological Euler characteristic
$\chi(\Hilb^m(X))$
has drawn much attention 
in connection with string theory. 
If $X$ is a (possibly singular) curve, then 
it is related to BPS state counting~\cite{PT3} and 
HOMFLY polynomials for links~\cite{SV}, \cite{MH}. 
If $X$ is a non-singular surface, then $\Hilb^m(X)$ is 
also non-singular and we have the 
remarkable formula by G\"ottsche~\cite{Got}
\begin{align}\label{intro:1}
\sum_{m\ge 0} \chi(\Hilb^m(X))q^{m-\frac{\chi(X)}{24}}
=\eta(q)^{-\chi(X)}. 
\end{align}
Here $\eta(q)$ is the Dedekind eta function
\begin{align}\label{eta}
\eta(q)=q^{\frac{1}{24}}\prod_{m\ge 1}(1-q^m). 
\end{align}
In particular, the generating series (\ref{intro:1})
is a Fourier development of a meromorphic modular form of weight $-\chi(X)/2$, 
which gives evidence of Vafa-Witten's S-duality conjecture~\cite{VW}
 in string theory. 
If $X$ is a smooth 3-fold, then $\chi(\Hilb^m(X))$
is related to the Donaldson-Thomas (DT) 
invariants~\cite{Thom}, \cite{MNOP}
and described in terms of 
MacMahon function~\cite{Li}, \cite{BBr}, \cite{LP}.  

Let $S$ be a \textit{singular} surface. 
In this case, $\chi(\Hilb^m(S))$ seems to 
be studied in few literatures\footnote{
In~\cite{GS3}, 
the weighted Euler characteristics of $\Hilb^m(S)$
for a K3 surface $S$
with $A_1$-type singularities is studied. 
The formula in~\cite[Example~3.26]{GS3}
involves Noether-Lefschetz numbers, 
and is different from ours in Theorem~\ref{thm:intro}.}.
Because of the singularities of $S$, 
the scheme $\Hilb^m(S)$ is no longer
non-singular and $\chi(\Hilb^m(S))$ reflects the 
complexity of the singularities of $S$. 
The behavior of the invariants 
$\chi(\Hilb^m(S))$ is more complicated than the smooth case, 
and it seems to be
difficult to see some good properties of their
generating series, e.g. 
the modularity.
Nevertheless we expect that the generating series of $\chi(\Hilb^m(S))$
has the modularity property 
as in the smooth case (\ref{intro:1}). 
This is motivated by a 3-fold version of the S-duality conjecture, 
stated as a modularity of the generating series of 
DT invariants on Calabi-Yau 3-folds 
counting two dimensional torsion 
sheaves on them with possibly singular supports. 
The purpose of this paper is to prove such a 
modularity for any singular surface $S$
with at worst $A_n$-type singularities, 
a simplest class of surface singularities.
It gives a first definitive result for 
the modularity of the generating series of $\chi(\Hilb^m(S))$
for a singular surface $S$.

\subsection{Main result}
Recall that 
an algebraic surface $S$
has an $A_n$-\textit{type singularity} at 
$p \in S$ if
the germ $(S, p)$ 
is analytically isomorphic to the affine singularity
\begin{align}\label{An}
A_n \cneq \{xy-z^{n+1}=0 : (x, y, z) \in \mathbb{C}^3\}
\end{align}
at the origin. 
The following is the main result
in this paper: 
\begin{thm}\label{thm:intro}
Let $S$ be a quasi-projective surface 
which is smooth except $A_{n_i}$-type singularities 
$p_i \in S$ for $1\le i\le l$. 
Then we have the following formula: 
\begin{align}\label{intro:2}
\sum_{m\ge 0} \chi(\Hilb^m(S))q^{m-\frac{\chi(\widetilde{S})}{24}}
=\eta(q)^{-\chi(\widetilde{S})} \cdot \prod_{i=1}^{l} \Theta_{n_i}(q). 
\end{align}
Here $\widetilde{S} \to S$ is the minimal resolution, 
and $\Theta_n(q)$ is defined by 
\begin{align}\label{Theta}
\Theta_n(q) \cneq \sum_{(k_1, \cdots, k_n) \in \mathbb{Z}^{n}}
q^{ \sum_{1\le i\le j \le n} k_i k_j}
e^{\frac{2\pi \sqrt{-1}}{n+2}(k_1 + 2k_2 + \cdots + nk_n)}. 
\end{align}
\end{thm}
By an elementary argument, 
we show that $\Theta_n(q)$ is 
a $\mathbb{Q}$-linear combination of the 
theta series determined by some 
integer valued positive definite 
quadratic forms 
on $\mathbb{Z}^n$ 
(cf.~Proposition~\ref{prop:modular} and Table~1 for small $n$). 
In particular, $\Theta_n(q)$ is a modular form of 
weight $n/2$, and we obtain the following corollary: 
\begin{cor}\label{cor:modular}
The generating series (\ref{intro:2})
is a Fourier development of a meromorphic 
modular form of weight $-\chi(S)/2$ 
for some congruence subgroup in 
$\mathrm{SL}_2(\mathbb{Z})$. 
\end{cor}
\subsection{Outline of the proof}
Here is an outline of the arguments:
in Section~\ref{sec:review}, 
we give a closed formula of 
the generating series of 
Euler characteristics of
rank one Quot schemes of 
points on $A_{n-1}$ in terms of an 
infinite product. 
Let $D \subset A_{n-1}$ be
the Weill divisor defined by 
\begin{align*}
D\cneq (x=z=0) \subset A_{n-1}
\end{align*}
in the notation of (\ref{An}). 
For $j\in \mathbb{Z}$, we denote
by $\oO_{A_{n-1}}(jD)$ the 
rank one reflexive sheaf
associated to the Weil divisor
$jD$. Note that any rank one reflexive sheaf on $A_{n-1}$ is 
isomorphic to $\oO_{A_{n-1}}(jD)$ for some $0\le j\le n-1$. 
Let
\begin{align}\label{Quotm0}
\Quot^{m}(\oO_{A_{n-1}}(jD))
\end{align}
be the Quot scheme 
which parametrizes the quotients $\oO_{A_{n-1}}(jD) \twoheadrightarrow Q$
where $Q$ is a zero dimensional coherent 
sheaf on $A_{n-1}$ with length $m$. 
Note that if $j=0$, the 
scheme (\ref{Quotm0}) coincides with 
$\Hilb^m(A_{n-1})$.
We will show the following formula
in Subsection~\ref{subsec:proof}:
\begin{align}\notag
&\sum_{\begin{subarray}{c}
0\le j\le n-1 \\
m\ge 0, k\in \mathbb{Z}
\end{subarray}}
\chi(\Quot^{m}(\oO_{A_{n-1}}(jD)))
q^{\frac{k^2n}{2}+\left(\frac{n}{2}-j \right)k+m}
t^{kn-j} \\
\label{Quotm}
&\hspace{50mm}=\prod_{m\in \mathbb{Z}_{>0}}
f_n(q^{m}t) \prod_{m\in \mathbb{Z}_{\ge 0}}
f_n(q^{m}t^{-1}). 
\end{align}
Here $f_n(x)$ is given by 
\begin{align}\label{defi:fn}
f_n(x)\cneq 1+ x+ \cdots + x^{n}.
\end{align}
The invariant $\chi(\Hilb^m(A_{n-1}))$ is 
obtained as a coefficient of $q^{m}t^0$ in the 
RHS of (\ref{Quotm}). 

Since $A_{n-1}$ is a toric surface, one may 
try to prove (\ref{Quotm}) via torus localization. 
It turns out that, by localization, the invariant 
$\chi(\Quot^m(\oO_{A_{n-1}}(jD)))$ coincides with the number of 
$n$-tuples of Young diagrams
satisfying a certain constraint
(cf.~Lemma~\ref{lem:Young}).  
However we are not able to 
prove
the formula (\ref{Quotm})
using the above combinatorial
description of $\chi(\Quot^m(\oO_{A_{n-1}}(jD)))$, 
nor find it in a literature, e.g. in the book~\cite{Stan}. 
The formula (\ref{Quotm}) is obtained by 
a rather indirect method:
it is a by-product of a flop transformation formula of 
DT type invariants counting two dimensional 
torsion sheaves on smooth
3-folds, established in the author's 
previous paper~\cite{TodS}.
The idea is as follows: 
let us consider a 3-fold flop 
$\phi \colon X \dashrightarrow X^{\dag}$ between 
smooth projective 3-folds, 
and a non-singular divisor $S \subset X$. 
Then it often happens that the strict transform 
$S^{\dag} \subset X^{\dag}$ of $S$ 
has singularities. 
We construct such a flop so that $S^{\dag}$ only 
has an $A_{n-1}$-type singularity. 
Then the flop formula in~\cite[Theorem~3.23]{TodS} compares invariants
counting rank one torsion free sheaves
on $S$ with those on $S^{\dag}$. As the former one 
is computed by (\ref{intro:1}), we obtain 
the formula which computes invariants 
counting rank one torsion free sheaves on $S^{\dag}$. 
Again using (\ref{intro:1}) for the smooth part of $S^{\dag}$, 
we obtain a contribution of the invariants 
from the singular point of $S^{\dag}$, 
which gives the formula (\ref{Quotm}). 

We note that the flop formula in~\cite{TodS}
relies on 
Bridgeland's equivalence of derived categories of
coherent sheaves under 3-fold
flops~\cite{Br1}, and the 
Hall algebra method
which is developed in recent years~\cite{JS}, \cite{K-S}, \cite{Tcurve1}, \cite{Tcurve2}, \cite{BrH}, \cite{Cala}. 
In turn, this 
indicates that the algebraic geometry involving
flops, 
derived categories and Hall algebras 
provides an interesting application 
to a study of enumerative combinatorics. 
Also it may be worth pointing out that the formula
(\ref{Quotm}) for $n=1$ together with (\ref{intro:1})
show that 
\begin{align*}
\sum_{k\in \mathbb{Z}}
q^{\frac{k^2}{2}+\frac{k}{2}}t^k
=\prod_{m\ge 1}(1-q^m) \prod_{m>0}(1+q^m t) \prod_{m\ge 0}(1+q^m t^{-1}). 
\end{align*}
The above formula is nothing but Jacobi triple product formula. 
It is surprising that the above classical result is 
also proved using 3-fold flops, derived categories, etc. 

In Section~\ref{sec:proof}, we prove 
Theorem~\ref{thm:intro}.
By a standard argument, the result is 
reduced to the case of $S=A_n$. 
In this case, the result 
follows by working with the formula (\ref{Quotm})
using Jacobi triple product formula. 
After that, we show the 
modularity of the series $\Theta_n(q)$ by 
describing $\Theta_n(q)$ as a $\mathbb{Q}$-linear 
combination of the theta series determined  by integer
valued positive definite quadratic forms. 
In Section~\ref{sec:Append}, as an appendix, 
we provide a combinatorial 
description of $\chi(\Quot^m(\oO_{A_{n-1}}(jD)))$
in terms of $n$-tuples of Young diagrams. 

The idea in this paper using the flop 
formula has possibilities to be applied 
for other surface singularities, 
but
we leave them for a future work. 
\subsection{Acknowledgment}
This work is supported by World Premier 
International Research Center Initiative
(WPI initiative), MEXT, Japan. This work is also supported by Grant-in Aid
for Scientific Research grant (22684002)
from the Ministry of Education, Culture,
Sports, Science and Technology, Japan.

\begin{table}[h]\label{table}
\caption{Descriptions of $\Theta_n(q)$ for $1\le n\le 4$}
\begin{align*}
&\Theta_1(q)=-\frac{1}{2}\sum_{k\in \mathbb{Z}} q^{k^2}
+\frac{3}{2}\sum_{k\in \mathbb{Z}} q^{9k^2}
 \\
&\Theta_2(q)=
-\sum_{(k_1, k_2) \in \mathbb{Z}^2}
q^{3k_1^2 + k_2^2 }
+2\sum_{(k_1, k_2) \in \mathbb{Z}^2}
q^{3k_1^2 + 4k_2^2}
 \\
&\Theta_3(q)=
-\frac{1}{4} \sum_{(k_1, k_2, k_3) \in \mathbb{Z}^3}
q^{k_1^2 + k_2^2 + k_3^2 + k_1 k_2 + k_1 k_3 + k_2 k_3} \\
&\hspace{40mm}
+\frac{5}{4} \sum_{(k_1, k_2, k_3) \in \mathbb{Z}^3}
q^{25k_1^2 +3k_2^2 + 7k_3^2 -15 k_1 k_2 -25 k_1 k_3 + 8k_2 k_3} \\
&\Theta_4(q)=\frac{1}{2}\sum_{(k_1, k_2, k_3, k_4) \in \mathbb{Z}^4}
q^{k_1^2 + k_2^2 + k_3^2 + k_4^2 +k_1 k_2 + k_2 k_3 + k_3 k_4 + k_1 k_3 + k_1 k_4 + k_2 k_4}\\
&-\sum_{(k_1, k_2, k_3, k_4) \in \mathbb{Z}^4}
q^{4k_1^2+3k_2^2 +7k_3^2+13k_4^2 -6 k_1 k_2 + 8k_2 k_3 + 18 k_3 k_4
-10 k_1 k_3 -14 k_1 k_4 + 11k_2 k_4} \\
&-\frac{3}{2}\sum_{(k_1, k_2, k_3, k_4) \in \mathbb{Z}^4}
q^{9k_1^2+3k_2^2 +7k_3^2+13k_4^2 -9 k_1 k_2 + 8k_2 k_3 + 18 k_3 k_4
-15 k_1 k_3 -21 k_1 k_4 + 11k_2 k_4} \\
&+3\sum_{(k_1, k_2, k_3, k_4) \in \mathbb{Z}^4}
q^{36k_1^2+3k_2^2 +7k_3^2+13k_4^2 -18 k_1 k_2 + 8k_2 k_3 + 18 k_3 k_4
-30 k_1 k_3 -42 k_1 k_4 + 11k_2 k_4}. 
\end{align*}
\end{table}

\section{Euler characteristics of Quot schemes of points on $A_{n-1}$}\label{sec:review}
\subsection{3-fold flops}
This subsection is devoted to a preliminary of 
the proof of the formula (\ref{Quotm}). 
We first
fix a 3-fold flop whose exceptional 
locus has \textit{width} $n$ in the sense of~\cite{Rei}, 
satisfying some properties. 
\begin{lem}\label{lem:flop}
For each $n\ge 1$, there 
exist smooth projective 3-folds 
$X$, $X^{\dag}$ and a flop diagram 
\begin{align}\label{intro:flop}
\xymatrix{
(C\subset X) \ar[dr]_{f}  \ar@{.>}[rr]^{\phi} &  & (X^{\dag} \supset C^{\dag})
 \ar[dl]^{f^{\dag}} \\
&  (p\in Y).  &
}
\end{align}
satisfying the following conditions: 
\begin{itemize}
\item There is a Zariski open neighborhood
$p\in U \subset Y$ which is isomorphic to the affine variety 
\begin{align}\label{affine:n}
\{ xy+z^2 -w^{2n}=0 : (x, y, z, w) \in \mathbb{C}^4\}. 
\end{align}
In particular, the exceptional locus of $f, f^{\dag}$ 
are irreducible rational curves $C, C^{\dag}$ 
which are contracted to
$p\in Y$. 
\item There is an irreducible 
smooth divisor $S\subset X$ 
such that $S \cap C$
is scheme theoretically 
one point, and\footnote{This last condition 
is required to make the computations of the Mukai 
vectors in Subsection~\ref{subsec:apply} simpler, and 
not essential.} $S^2=S^3=0$. 
\item The strict transform $S^{\dag} \subset X^{\dag}$
of $S$ contains $C^{\dag}$, 
 has a $A_{n-1}$-type singularity at a point $o\in S^{\dag}$, 
and $S^{\dag} \setminus \{o\}$ is smooth. 
\end{itemize} 
\end{lem}
\begin{proof}
We take 
$Y$ to be a projective compactification
of the affine variety (\ref{affine:n}) 
which is smooth outside $0 \in \mathbb{C}^4$. 
We take
a flop diagram (\ref{intro:flop})
by blowing up at the 
Weil divisors on $Y$, given by 
the closures of the subschemes
\begin{align*}
(x=z+w^n=0) \subset U, \quad (x=z-w^n=0) \subset U
\end{align*}
respectively. 
 By the construction, there exists a divisor $T \subset X$
with $T\cap C$ scheme theoretically one point. 
Let $H_Y$ be a sufficiently ample divisor on $Y$. 
Then $T + f^{\ast}H_Y$ is ample and globally generated
by the base point free theorem. 
Let 
\begin{align*}
T_1, T_2 \in \lvert T + f^{\ast}H_Y \rvert
\end{align*}
be general members. We 
replace $X, X^{\dag}, Y$ by blow-ups at $T_1 \cap T_2$
which is smooth and lies outside $C, C^{\dag}, p$
respectively.  
Then by setting $S$ to be the 
connected component of the strict 
transform of $T_1$ which 
intersects with $C$, we obtain a 
diagram (\ref{intro:flop}) satisfying the first and the second conditions. 

The last statement
can be directly checked by 
describing the birational map $\phi$ on each 
affine charts of crepant resolutions of (\ref{affine:n}). 
An alternative geometric argument is as follows: 
by~\cite{Rei}, 
the birational map $\phi$
is given by the Pagoda diagram, 
\begin{align}\label{Pagoda}
X \stackrel{f_1}{\leftarrow} X_1 \stackrel{f_2}{\leftarrow}
 \cdots \stackrel{f_{n-1}}{\leftarrow} 
X_{n-1} \stackrel{f_n}{\leftarrow} X_n \stackrel{f_n^{\dag}}{\to}
 X_{n-1}^{\dag} \stackrel{f_{n-1}^{\dag}}{\to} \cdots 
\stackrel{f_2^{\dag}}{\to} X_1^{\dag} 
\stackrel{f_1^{\dag}}{\to} 
X^{\dag}. 
\end{align}
Here $f_i$, $f_i^{\dag}$ for $1\le i\le n-1$
are blow-ups at 
$(0, -2)$-curves, and $f_{n}$, $f_n^{\dag}$ are blow-ups
at $(-1, -1)$-curves. 
Hence 
the 
birational map $S \dashrightarrow S^{\dag}$
decomposes into 
\begin{align*}
S=S_0 \stackrel{g_1}{\leftarrow} S_1 \stackrel{g_2}{\leftarrow}
 \cdots \stackrel{g_{n-1}}{\leftarrow} 
S_{n-1} \stackrel{g_n}{\leftarrow} S_n \stackrel{g_n^{\dag}}{\to}
 S_{n-1}^{\dag} \stackrel{g_{n-1}^{\dag}}{\to} \cdots 
\stackrel{g_2^{\dag}}{\to} S_1^{\dag} 
\stackrel{g_1^{\dag}}{\to} S_0^{\dag}=
S^{\dag}. 
\end{align*}
Here each $g_i \colon S_i \to S_{i-1}$ is a blow-up 
at a point in $\Ex(g_{i-1}) \setminus 
g_{i-1 \ast}^{-1}\Ex(g_{1} \circ 
\cdots \circ g_{i-2})$, where $g_{i-1 \ast}^{-1}$ is the strict 
transform. 
The exceptional locus of $S_n \to S$ is a $A_{n-1}$-configuration of 
$(-2)$-curves together with a tail of a $(-1)$-curve, given by 
$\Ex(g_n)$. 
The birational morphism $S_n \to S^{\dag}$
contracts the above $A_{n-1}$-configuration
of $(-2)$-curves on $S_n$ to 
a $A_{n-1}$-singularity $o\in S^{\dag}$, and 
the image of the tail $\Ex(g_n)$ 
coincides with $C^{\dag}$. 
\end{proof}
In what follows, we fix a flop diagram (\ref{intro:flop}). 
We next describe rank one torsion free sheaves on $S$ and $S^{\dag}$:
\begin{lem}\label{lem:sheaves}
(i) An object $E \in \Coh(S)$ is a rank 
one torsion free sheaf with trivial determinant 
on $S \setminus C$ if and only
if $E$ is an ideal sheaf $I_Z$ for some 
zero dimensional subscheme $Z \subset S$. 

(ii) 
An object $E \in \Coh(S^{\dag})$
is a rank one torsion free sheaf 
if and only if it fits into the exact sequence
\begin{align}\label{ELQ}
0 \to E \to \lL \to Q \to 0
\end{align}
where $\lL$ is a rank one reflexive 
sheaf on $S^{\dag}$, $Q$ is a zero dimensional 
sheaf on $S^{\dag}$. 
Moreover $E$ has a trivial determinant on 
$S^{\dag}\setminus C^{\dag}$ if and only if 
$\lL$ is of the form $\oO_{S^{\dag}}(jC^{\dag})$
for some $j\in \mathbb{Z}$. 
\end{lem}
\begin{proof}
The result follows from a well-known argument. 
As for (ii), let 
$E \in \Coh(S^{\dag})$ be a rank one 
torsion free sheaf. 
We have the exact sequence in $\Coh(S^{\dag})$
\begin{align*}
0 \to E \to E^{\vee \vee} \to Q \to 0. 
\end{align*}
Since $S^{\dag}$ is normal, $Q$ is a zero dimensional 
sheaf. By setting $\lL=E^{\vee \vee}$, 
we obtain the exact sequence (\ref{ELQ}).
Conversely if $E$ fits into (\ref{ELQ}), 
then 
obviously $E$ is a rank one torsion free sheaf. 
The last assertion is also obvious. 
\end{proof}

\subsection{Application of the flop formula}\label{subsec:apply}
Let us consider a flop diagram (\ref{intro:flop}). 
We denote by $i$, $i^{\dag}$ the 
closed embeddings $S \subset X$, $S^{\dag} \subset X^{\dag}$
respectively, and 
fix an ample divisor $\omega$ on $Y$. 
The flop transformation formula 
of DT type invariants in~\cite{TodS}
compares 
invariants counting 
$f^{\ast}\omega$-semistable torsion sheaves on $X$
supported on $S$
with 
those counting 
$f^{\dag \ast}\omega$-semistable torsion sheaves on $X^{\dag}$
supported on $S^{\dag}$. 
For $\beta \in H_2(X)$ and $\gamma\in \mathbb{Q}$, 
let $M_{\beta, \gamma}(S)$ be the moduli space of 
rank one torsion free sheaves $E$ on $S$ 
such that the Mukai vector of $i_{\ast}E$ satisfies
\begin{align}\label{Mukai}
\ch(i_{\ast}E)\sqrt{\td_X}&=(0, S, -\beta, -\gamma) \\
\notag
&\in H^0(X) 
\oplus H^2(X) \oplus H^4(X) \oplus H^6(X). 
\end{align}
Here we have identified $H^4(X)$, $H^6(X)$
with $H_2(X)$, $\mathbb{Q}$ by the
Poincar\'e duality. 
We note that the 
$f^{\ast}\omega$-semistable 
sheaves on $X$ supported $S$ 
with Mukai vector $(0, S, -\beta, -\gamma)$
coincide with the sheaves $i_{\ast}E$\
for $[E] \in M_{\beta, \gamma}(S)$.
The similar statement also holds for $f^{\dag \ast}\omega$-semistable
sheaves on $X^{\dag}$ supported on $S^{\dag}$. 
Therefore in this situation, 
 the flop formula in~\cite[Theorem~3.23 (ii)]{TodS} is described as 
\begin{align}\notag
\sum_{\beta^{\dag} \in H_2(X^{\dag}), \gamma\in \mathbb{Q}}
\chi(M_{\beta^{\dag}, \gamma}(S^{\dag}))q^{\gamma} t^{\beta^{\dag}}
&= \sum_{\beta\in H_2(X), \gamma\in \mathbb{Q}} 
\chi(M_{\beta, \gamma}(S))q^{\gamma} t^{\phi_{\ast}\beta} \\
&\label{flop:form}\cdot q^{\frac{n}{12}}t^{\frac{n}{2}C^{\dag}} \prod_{m\in \mathbb{Z}_{>0}}
f_n(q^m t^{C^{\dag}}) \prod_{m\in \mathbb{Z}_{\ge 0}} f_n(q^m t^{-C^{\dag}}). 
\end{align}
Here $f_n(x)$ is the polynomial (\ref{defi:fn}). 
The formula (\ref{flop:form})
also holds after replacing $M_{\beta, \gamma}(S)$, $M_{\beta^{\dag}, \gamma}(S^{\dag})$
by the subschemes
\begin{align*}
M_{\beta, \gamma}'(S) \subset 
M_{\beta, \gamma}(S), \ M_{\beta^{\dag}, \gamma}'(S^{\dag})
\subset M_{\beta^{\dag}, \gamma}(S^{\dag})
\end{align*}
consisting of $[E] \in M_{\beta, \gamma}(S)$, $[E^{\dag}] \in M_{\beta^{\dag}, \gamma}(S^{\dag})$ which 
have \textit{trivial determinants}
 on $S \setminus C$, $S^{\dag} \setminus C^{\dag}$
respectively. 
(Indeed in the proof of~\cite[Theorem~3.23]{TodS}, 
it is enough to notice that 
$E \in \pB_{f^{\ast}\omega}^{\mu, S}$
has a trivial determinant on $S \setminus C$
if and only the same holds for $F_2 \in \cC_{f^{\ast}\omega}^{\mu, S}$.)
By Lemma~\ref{lem:sheaves} (i), 
the objects which contribute to $\chi(M_{\beta, \gamma}'(S))$ are 
of the form
\begin{align}\label{IZ}
I_Z \subset \oO_S, \quad Z \subset S
\end{align}
where $Z$ is a zero dimensional subscheme and
$I_Z$ is the ideal sheaf of $Z$. 
Also by Lemma~\ref{lem:sheaves} (ii), the 
objects which contribute to 
$\chi(M_{\beta^{\dag}, \gamma}'(S^{\dag}))$ are of the form
\begin{align}\label{Quo}
\Ker(\oO_{S^{\dag}}(lC^{\dag}) \twoheadrightarrow Q), \quad
 l\in \mathbb{Z}
\end{align}
where $Q$ is a zero dimensional sheaf on $S^{\dag}$. 
We need to compute the 
 Mukai vectors of the push-forward of (\ref{IZ}), (\ref{Quo})
to $X$, $X^{\dag}$. 
As for (\ref{IZ}), 
it is easily computed as 
\begin{align*}
\left(0, S, \frac{c_1(X)}{4}S, \frac{c_1(X)^2}{96}S + 
\frac{c_2(X)}{24}S -\lvert Z \rvert   \right)
\end{align*}
using
the condition $S^2=S^3=0$ in Lemma~\ref{lem:flop}, 
the resolution 
\begin{align*}
0 \to \oO_X(-S) \to \oO_X \to \oO_S \to 0
\end{align*}
and
\begin{align}\label{tdX}
\sqrt{\mathrm{td}_X}=
\left(1, \frac{c_1(X)}{4}, \frac{c_2(X)}{24} + \frac{c_1(X)^2}{96}, 
\frac{c_1(X) c_2(X)}{96}-\frac{c_1(X)^3}{384}\right). 
\end{align}
As for (\ref{Quo}), it requires some more 
arguments: 
\begin{lem}\label{lem:Mukai}
For $0\le j\le n-1$, we have 
\begin{align*}
\ch(i_{\ast}^{\dag}\oO_{S^{\dag}}(jC^{\dag}))
=\left(0, S^{\dag}, 
\left( j-\frac{n}{2} \right)C^{\dag}, 
-\frac{n}{6}   \right). 
\end{align*}
\end{lem}
\begin{proof}
We first recall Bridgeland's 
perverse coherent sheaves~\cite{Br1}
(also see~\cite[Subsection~2.2]{TodS}) defined by 
\begin{align*}
\iPPer_0(X^{\dag}/Y) \cneq
\left\{ E \in D^b \Coh(X^{\dag}) : 
\begin{array}{c}
\dR f_{\ast}^{\dag}E \in \Coh_0(Y) \\
\Hom^{<1}(E, \oO_{C^{\dag}}(-1))=0 \\
\Hom^{<-1}(\oO_{C^{\dag}}(-1), E)=0
\end{array}
\right\}.
\end{align*}
Here $\Coh_0(Y)$ is the category of zero 
dimensional sheaves on $Y$. 
We set $\iT$ and $\iF$ to be
\begin{align*}
 \iT \cneq 
\iPPer_0(X^{\dag}/Y) \cap \Coh(X^{\dag}), \ 
\iF \cneq 
\iPPer_0(X^{\dag}/Y)[-1] \cap \Coh(X^{\dag}). 
\end{align*}
Then by~\cite{MVB}, 
$(\iT, \iF)$ forms a torsion pair 
of the category
\begin{align*}
\Coh_0(X^{\dag}/Y) \cneq \{ F \in \Coh(X^{\dag}) : 
f_{\ast}^{\dag} F \in \Coh_0(Y)\}
\end{align*}
and $\iPPer_0(X^{\dag}/Y)$ is the associated tilting, i.e. 
$\langle \iF [1], \iT \rangle$.
 
For $k\ge 1$, let 
$kC^{\dag}\subset S^{\dag}$ be the subscheme
defined by the ideal $\oO_{S^{\dag}}(-kC^{\dag}) \subset \oO_{S^{\dag}}$. 
We prove that $\chi(\oO_{kC^{\dag}})=k$ holds for $1\le k\le n$. 
Since there is a surjection 
$\oO_{X^{\dag}} \twoheadrightarrow \oO_{kC^{\dag}}$, 
it follows that
\begin{align*}
\oO_{kC^{\dag}} \in \iT.
\end{align*}
 Note that, since $\oO_{f^{\dag -1}(p)}=\oO_{C^{\dag}}$, 
any one dimensional stable sheaf on $X^{\dag}$
supported on $C^{\dag}$ must be
of the form $\oO_{C^{\dag}}(a)$ for some $a\in\mathbb{Z}$. 
By~\cite[Lemma~2.4]{TodS}, 
the category $\iF$ is the extension closure of
objects of the form $\oO_{C^{\dag}}(a)$
with $a\le -1$. 
Since $(\iT, \iF)$ is a torsion pair of $\Coh_0(X^{\dag}/Y)$, 
this implies that the stable factors of 
$\oO_{kC^{\dag}}$ consist of $\oO_{C^{\dag}}(a_i)$
for $1\le i\le k$ with $a_i \ge 0$. 
In particular, we have the inequality 
\begin{align*}
\chi(\oO_{kC^{\dag}}) \ge k
\end{align*} and the equality 
holds if and only if $a_i=0$ for all $i$. 

On the other hand
since $S^2=0$,  
a local computation
easily shows that 
$\oO_{S^{\dag}}(S^{\dag}) \cong 
\oO_{S^{\dag}}(nC^{\dag})$. 
Hence we have 
\begin{align*}
\ch(\oO_{nC^{\dag}})=\ch(\oO_{S^{\dag}})- \ch(\oO_{S^{\dag}}(-S^{\dag}))
\end{align*}
which shows $\chi(\oO_{nC^{\dag}})=-S^{\dag 3}$.
Since $S^{\dag 3}=nC^{\dag} \cdot S^{\dag}=-n$, 
we obtain $\chi(\oO_{nC^{\dag}})=n$.
Hence the above argument 
shows that $\oO_{nC^{\dag}}$
is a $n$-step extensions of $\oO_{C^{\dag}}$. 
Let $\oO_{nC^{\dag}} \twoheadrightarrow T_k$ be a 
surjection such that $T_k$ is a $k$-step extensions of $\oO_{C^{\dag}}$. 
Then $T_k$ is a structure sheaf of a pure one dimensional subscheme 
$kC^{'\dag} \subset S^{\dag}$
with fundamental cycle $k[C^{\dag}]$.  
Since this is a characterizing property of $kC^{\dag}$, 
we have $kC^{\dag}=kC^{'\dag}$. 
Therefore 
we have $\oO_{kC^{\dag}} \cong T_k$, and
$\chi(\oO_{kC^{\dag}})=k$ holds. 

The above computation shows that
\begin{align*}
\ch(i_{\ast}^{\dag}\oO_{S^{\dag}}(-kC^{\dag}))=
\left(0, S^{\dag}, -\frac{S^{\dag 2}}{2} -kC^{\dag}, 
\frac{S^{\dag 3}}{6}-k  \right)
\end{align*}
for $1\le k\le n$. 
Setting $j=n-k$ and noting 
$S^{\dag 2}=nC^{\dag}$, $S^{\dag 3}=-n$, 
$i_{\ast}^{\dag}\oO_{S^{\dag}}(jC^{\dag})
=i_{\ast}^{\dag}\oO_{S^{\dag}}(-k C^{\dag}) \otimes \oO_{X^{\dag}}(S^{\dag})$, 
we obtain the result. 
\end{proof}
We write 
$l=kn+j$ for 
$k\in \mathbb{Z}$ and $0\le j\le n-1$.
Note that we have 
\begin{align*}
\ch(i_{\ast}^{\dag}\oO_{S^{\dag}}((kn+j)C^{\dag}))
= e^{kS^{\dag}}(\ch(i_{\ast}^{\dag}\oO_{S^{\dag}}(jC^{\dag}))). 
\end{align*}
Together with Lemma~\ref{lem:Mukai} and (\ref{tdX}), 
a little computation shows that 
the Mukai vector of $i_{\ast}^{\dag}$ of (\ref{Quo}) 
is computed as 
\begin{align*}
& \\
&\left(0, S^{\dag}, \left(kn + j -\frac{n}{2} \right)C^{\dag} +
 \frac{c_1(X^{\dag})}{4}S^{\dag}, \right. \\
\notag
&\hspace{20mm}\left. -\frac{n}{6} -kj + \frac{kn}{2}-\frac{k^2 n}{2}
+ \left( \frac{c_1(X^{\dag})^2}{96} + \frac{c_2(X^{\dag})}{24}
\right)S^{\dag}  -\lvert Q \rvert \right). 
\end{align*}

\subsection{Proof of the formula (\ref{Quotm})}\label{subsec:proof}
\begin{proof}
For a variety $X$ and $\lL \in \Coh(X)$, 
we denote by $\Quot^m(\lL)$ the
Quot scheme which parametrizes the zero dimensional quotients 
\begin{align}\label{LQ}
\lL \twoheadrightarrow Q, \quad \length Q=m.
\end{align}
Also for $p\in X$, we denote by 
\begin{align*}
\Quot_p^m(\lL) \subset \Quot^m(\lL)
\end{align*}
the subscheme 
consisting of quotients (\ref{LQ}) such that 
$Q$ is supported on $p$. 
By (\ref{flop:form}) and the arguments in the previous subsection, we obtain 
\begin{align}\notag
&\sum_{\begin{subarray}{c}
m\ge 0, 
k \in \mathbb{Z} \\
0 \le j\le n-1
\end{subarray}}
\chi(\Quot^m(\oO_{S^{\dag}}((kn+j)C^{\dag})))q^{m+\frac{n}{6} +kj - 
\frac{kn}{2}+\frac{k^2 n}{2}
-\frac{c_2(X^{\dag})}{24}S^{\dag}}t^{\left(\frac{n}{2}-j-kn\right)C^{\dag}} \\
\notag
&=\sum_{m\ge 0} \chi(\Hilb^m(S))q^{m-\frac{c_2(X)}{24}S}
\cdot 
q^{\frac{n}{12}}t^{\frac{n}{2}C^{\dag}}
\prod_{m \in \mathbb{Z}_{>0}}f_n(q^{m}t^{C^{\dag}})
\prod_{m \in \mathbb{Z}_{\ge 0}} f_n(q^{m} t^{-C^{\dag}}).
\end{align}
Here we have used that  
\begin{align*}
\phi_{\ast}\left( \frac{c_1(X)}{4}S, 
\frac{c_1(X)^2}{96}S \right)=\left( 
\frac{c_1(X^{\dag})}{4}S^{\dag}, 
\frac{c_1(X^{\dag})^2}{96}S^{\dag} \right)
\end{align*}
since $c_1(X)$ and $c_1(X^{\dag})$ are 
pull-backs from divisor classes on $Y$. 

We simplify both sides of the above equation. 
Since $S^{\dag 2}=nC^{\dag}$, we have the isomorphism
\begin{align}\notag
\otimes \oO_{X^{\dag}}(kS^{\dag}) \colon 
\Quot^m(\oO_{S^{\dag}}(jC^{\dag}))
&\stackrel{\cong}{\to}
\Quot^m(\oO_{S^{\dag}}((kn+j)C^{\dag})).
\end{align}
Hence the Euler characteristics of both sides coincide.
Also note that the Weil
divisor $D \subset A_{n-1}$ 
corresponds to $C^{\dag} \subset S^{\dag}$ 
under a local isomorphism between
$0 \in A_{n-1}$ and 
$o \in S^{\dag}$. 
Hence we have an isomorphism 
\begin{align}
\label{isom:AQ}
\Quot_{o}^m(\oO_{S^{\dag}}(jC^{\dag}))
&\cong \Quot_{0}^m(\oO_{A_{n-1}}(jD)). 
\end{align}
We also have the stratification
\begin{align}\label{strata}
&\Quot^m(\oO_{S^{\dag}}(jC^{\dag})) \\
&\notag=\coprod_{m_1+m_2=m}
\Quot^{m_1}_{o}(\oO_{S^{\dag}}(jC^{\dag})) \times 
\Hilb^{m_2}(S^{\dag} \setminus \{o\}).
\end{align}
Combined these, we have the following equalities:
\begin{align*}
&\sum_{m\ge 0} \chi(\Quot^m(\oO_{S^{\dag}}(jC^{\dag})))q^m \\
&= \sum_{m\ge 0} \chi(\Quot^m_{o}(\oO_{S^{\dag}}(jC^{\dag})))q^m
\cdot \sum_{m\ge 0} \chi(\Hilb^m(S^{\dag} \setminus \{o\})) q^m \\
&=\sum_{m\ge 0} \chi(\Quot_{0}^m(\oO_{A_{n-1}}(jD)))q^m
\cdot \sum_{m\ge 0} \chi(\Hilb^m(S)) q^m \\
&=\sum_{m\ge 0} \chi(\Quot^m(\oO_{A_{n-1}}(jD)))q^m
\cdot \sum_{m\ge 0} \chi(\Hilb^m(S)) q^m. 
\end{align*}
Here the first equality follows from (\ref{strata}), 
the second equality follows from 
G$\ddot{\rm{o}}$ttsche formula (\ref{intro:1}), 
$\chi(S^{\dag} \setminus \{o\})=\chi(S)$
and (\ref{isom:AQ}), 
and the last equality follows from the torus localization on $A_{n-1}$. 
Also an easy computation (cf.~\cite[Lemma~2.8, Proposition~2.9]{TodS}) 
shows that 
 \begin{align*}
c_2(X) \cdot S=c_2(X^{\dag}) \cdot S^{\dag}-2n.
\end{align*}
Summing up, we arrive at the formula: 
\begin{align}\notag
&\sum_{\begin{subarray}{c}
0\le j\le n-1 \\
m\ge 0, k\in \mathbb{Z}
\end{subarray}}
\chi(\Quot^{m}(\oO_{A_{n-1}}(jD)))
q^{\frac{k^2n}{2}+\left(j-\frac{n}{2} \right)k+m}
t^{-\left(kn+j\right)C^{\dag}} \\
\notag
&\hspace{40mm}=\prod_{m\in \mathbb{Z}_{>0}}
f_n(q^{m}t^{C^{\dag}}) \prod_{m\in \mathbb{Z}_{\ge 0}}
f_n(q^{m}t^{-C^{\dag}}). 
\end{align}
By replacing $k$ by $-k$, we obtain the desired formula (\ref{Quotm}). 
\end{proof}
The following is an obvious corollary of the formula (\ref{Quotm}): 
\begin{cor}
We have the following formula: 
\begin{align}\label{intro:form}
\sum_{m\ge 0} \chi(\Hilb^m(A_{n}))q^m 
= \mathrm{Coeff}_{t^0}
\left(\prod_{m>0} f_{n+1}(q^m t) \prod_{m\ge 0} f_{n+1}(q^m t^{-1})\right). 
\end{align}
Here $\mathrm{Coeff}_{t^0}(\ast)$ means 
that taking the $t^0$ coefficient
of the formal series $\ast$ with variables
$q, t$, and $f_n(x)$ is given by (\ref{defi:fn}). 
\end{cor}
\section{Proof of the main result}\label{sec:proof}
\subsection{Proof of Theorem~\ref{thm:intro}}
\begin{proof}
In what follows, we set
\begin{align*}
\xi_m \cneq e^{\frac{2\pi \sqrt{-1}}{m}} \in \mathbb{C}. 
\end{align*}
We have the decomposition
\begin{align*}
f_n(x)=\prod_{i=1}^{n}(1-x\xi_{n+1}^{i}). 
\end{align*}
Therefore the RHS of (\ref{intro:form}) 
coincides with the $t^0$-coefficient of 
\begin{align}\label{pf1}
\prod_{i=1}^{n+1} \left(\prod_{m>0} (1-q^m t \xi_{n+2}^i) \prod_{m\ge 0}
(1-q^m t^{-1} \xi_{n+2}^{-i})\right). 
\end{align}
Using the Jacobi triple product formula
\begin{align*}
\sum_{k\in \mathbb{Z}}
q^{\frac{k^2}{2}+ \frac{k}{2}}(-t)^k=
\prod_{m\ge 1}(1-q^m) \prod_{m>0} (1-q^m t) \prod_{m\ge 0}(1-q^m t^{-1})
\end{align*}
the product (\ref{pf1}) is written as
\begin{align*}
\prod_{m\ge 1}(1-q^m)^{-n-1}\prod_{i=1}^{n+1}
\left(\sum_{k\in \mathbb{Z}}q^{\frac{k^2}{2}+ \frac{k}{2}}
(-t\xi^i_{n+2})^k  \right). 
\end{align*}
The $t^0$-coefficient of the above product becomes
\begin{align*}
\prod_{m\ge 1}(1-q^m)^{-n-1}
\cdot \left(\sum_{\begin{subarray}{c}
(k_1, \cdots, k_{n+1}) \in \mathbb{Z}^{n+1}\\
k_1 + \cdots + k_{n+1}=0
\end{subarray}}
q^{\frac{k_1^2}{2}+ \cdots + \frac{k_{n+1}^2}{2}}
\xi_{n+2}^{k_1 + 2k_2 + \cdots + (n+1)k_{n+1}}  \right). 
\end{align*}
The right sum coincides with $\Theta_n(q)$
defined by (\ref{Theta}). Therefore we obtain 
\begin{align}\label{form:A}
\sum_{m\ge 0} \chi(\Hilb^m(A_n))q^n =\prod_{m\ge 1}(1-q^m)^{-n-1}
\cdot \Theta_n(q). 
\end{align}
For a variety $X$ and $p\in X$, we denote by 
$\Hilb_p^m(X)$ the subscheme of $\Hilb^m(X)$
corresponding to the zero dimensional subschemes $Z \subset X$
with $\Supp(Z)=\{p\}$. 
Let $S$ be an algebraic surface as in Theorem~\ref{thm:intro}.
We have the stratification
\begin{align*}
\Hilb^m(S)=\coprod_{m_0+ m_1 + \cdots + m_l=m}
\Hilb^{m_0}(S^{o}) \times \prod_{i=1}^{l}
\Hilb^{m_i}_{p_i}(S). 
\end{align*}
Here $S^{o} \subset S$ is the smooth part of $S$. 
Noting that $p_i$ is an $A_{n_i}$-type singularity, the 
torus localization on $A_{n_i}$ 
shows that
\begin{align*}
\chi(\Hilb_{p_i}^{m}(S))=\chi(\Hilb_{0}^m(A_{n_i}))
=\chi(\Hilb^m(A_{n_i})). 
\end{align*}
Combined with (\ref{intro:1}) and (\ref{form:A}), 
we obtain the formula: 
\begin{align*}
\sum_{m\ge 0} \chi(\Hilb^m(S))q^m=
\prod_{m\ge 1}(1-q^m)^{-\chi(S^{o})-\sum_{i=1}^{l}(n_i+1)}
\cdot \prod_{i=1}^{l} \Theta_{n_i}(q). 
\end{align*}
For the minimal resolution $\widetilde{S} \to S$, we have
\begin{align*}
\chi(\widetilde{S})=\chi(S^{o}) + \sum_{i=1}^{l}(n_i +1). 
\end{align*}
Combined with the definition of $\eta(q)$ in (\ref{eta}), 
we obtain the desired formula (\ref{intro:2}). 
\end{proof}

\subsection{Modularity of $\Theta_n(q)$}
In order to conclude Corollary~\ref{cor:modular}, we need to 
check the modularity of $\Theta_n(q)$. 
Indeed, we show that 
$\Theta_n(q)$ is a $\mathbb{Q}$-linear combination of
the theta series determined by integer valued positive definite 
quadratic forms on $\mathbb{Z}^n$. 
For a positive definite quadratic form 
\begin{align*}
Q \colon \mathbb{Z}^n \to \mathbb{Z}
\end{align*}
let $\Theta_{Q}(q)$ be the associated theta series
\begin{align*}
\Theta_{Q}(q) \cneq \sum_{(k_1, \cdots, k_n) \in \mathbb{Z}^n}
q^{Q(k_1, \cdots, k_n)}. 
\end{align*}
It is well-known that $\Theta_Q(q)$ is a modular form of weight $n/2$
for some congruence subgroup in $\mathrm{SL}_2(\mathbb{Z})$
(cf.~\cite[Section~3.2]{Z123}). 
\begin{prop}\label{prop:modular}
The series
$\Theta_{n}(q)$ is a $\mathbb{Q}$-linear combination of 
the theta series determined by integer valued positive definite 
quadratic forms on $\mathbb{Z}^n$, i.e. 
there exist $N\ge 1$, $a_i \in \mathbb{Q}$
and integer valued positive definite 
quadratic forms $Q_i$ on $\mathbb{Z}^n$
for $1\le i\le N$ such that $\Theta_{n}(q)$ is written as
\begin{align*}
\Theta_n(q)=\sum_{i=1}^{N} a_i \Theta_{Q_i}(q). 
\end{align*}
\end{prop}
The result of Corollary~\ref{cor:modular}
follows from Theorem~\ref{thm:intro} together with the above
proposition. 
In order to prove Proposition~\ref{prop:modular}, we 
note the following: 
\begin{itemize}
\item By the base change of $\mathbb{Z}^n$ given by 
$k_1 \mapsto k_1 + 2k_2 + \cdots + nk_n$, $k_i \mapsto k_i$
for $i\ge 2$, the series $\Theta_n(q)$ is written as 
\begin{align*}
\Theta_n(q)=\sum_{(k_1, \cdots, k_n) \in \mathbb{Z}^n}
q^{Q(k_1, \cdots, k_n)}\xi^{k_1}_{n+2}
\end{align*}
for some integer valued positive definite 
quadratic form $Q$ on $\mathbb{Z}^n$. 
\item The series $\Theta_n(q)$ is invariant after
replacing $\xi_{n+2}$ 
by $g(\xi_{n+2})$ for any element 
$g \in \mathrm{Gal}(\mathbb{Q}(\xi_{n+2})/\mathbb{Q})$. 
This follows since the product expansion (\ref{pf1}) also holds
after replacing $\xi_{n+2}$ by $g(\xi_{n+2})$. 
\end{itemize}
Therefore the result of Proposition~\ref{prop:modular}
follows from the following proposition: 
\begin{prop}\label{prop:rest}
Let $n, m$ be the positive integers, and $Q$ an integer
valued positive definite 
quadratic form on $\mathbb{Z}^n$. 
Suppose that the series
\begin{align*}
\Theta_{Q, m}(q) =\sum_{(k_1, \cdots, k_n) \in \mathbb{Z}^n}
q^{Q(k_1, \cdots, k_n)} \xi_m^{k_1}
\end{align*}
is invariant after replacing $\xi_m$ by 
$g(\xi_m)$ for any $g\in \mathrm{Gal}(\mathbb{Q}(\xi_m)/\mathbb{Q})$. 
Then
$\Theta_{Q, m}(q)$ is a $\mathbb{Q}$-linear combination of the
theta series determined by integer valued positive definite 
quadratic forms on $\mathbb{Z}^n$. 
\end{prop}
The rest of this section is devoted to proving Proposition~\ref{prop:rest}.
\subsection{K-group of subsets of $\mathbb{Z}^n$} 
In what follows, we fix the notation in Proposition~\ref{prop:rest}.
We define the K-group of the subsets in $\mathbb{Z}^n$
to be
\begin{align*}
K^n \cneq \bigoplus_{ T \subset \mathbb{Z}^n} \mathbb{Z}[T]/\sim.
\end{align*}
Here the relation $\sim$ is generated by
\begin{align}\label{relation}
[T_1]+[T_2] \sim [T_1 \cup T_2] -[T_1 \cap T_2].  
\end{align}
For any element
\begin{align*}
\alpha=\sum_{i} a_i[T_i] \in K^n, \quad a_i \in \mathbb{Z}
\end{align*}
the series
\begin{align}\label{Qa}
\Theta_{Q, \alpha}(q) =\sum_{i}a_i \sum_{(k_1, \cdots, k_n) \in T_i}
q^{Q(k_1, \cdots, k_n)}
\end{align}
is well-defined as it respects the relation (\ref{relation}). 
Let 
\begin{align*}
1=m_1<m_2< \cdots <m_l=m
\end{align*}
be the set of divisors of $m$. 
We define the following subsets: 
\begin{align*}
&S_i \cneq \{ (k_1, \cdots, k_n) \in \mathbb{Z}^n : 
\mathrm{g.c.d.}(k_1, m)=m_i\} \\
&T_i \cneq \{ (k_1, \cdots, k_n) \in \mathbb{Z}^n : 
m_i | k_1\}. 
\end{align*}
\begin{lem}\label{lem:cont}
The element $[S_i] \in K^n$ is 
contained in the subgroup of $K^n$ generated by 
$[T_1], \cdots, [T_l]$. 
\end{lem}
\begin{proof}
We prove the claim by the induction on $i$. 
For $i=l$, we have $S_l=T_l$, and the 
statement is obvious. 
Suppose that the claim holds for 
$[S_j]$ with $j>i$. 
We have $S_i \subset T_i$
and the complement is 
the disjoint union of $S_j$
with $j>i$, $m_i|m_j$. 
Therefore we obtain
\begin{align*}
[S_i]=[T_i]-\sum_{j>i, m_i|m_j}[S_j]. 
\end{align*}
By the induction, the claim also holds for $[S_i]$. 
\end{proof}
\begin{lem}\label{lem:both}
Both of $\Theta_{Q, T_i}(q)$, $\Theta_{Q, S_i}(q)$ 
are $\mathbb{Z}$-linear combinations of the theta 
series determined by integer valued positive definite
quadratic forms on 
$\mathbb{Z}^n$. 
\end{lem}
\begin{proof}
The claim for $\Theta_{Q, T_i}(q)$ is obvious 
since
\begin{align*}
\Theta_{Q, T_i}(q)=\Theta_{Q_i}(q), \quad 
Q_i(k_1, \cdots, k_n)=Q(m_i k_1, k_2, \cdots, k_n). 
\end{align*}
The claim for $\Theta_{Q, S_i}(q)$ follows from 
the claim for $\Theta_{Q, T_i}(q)$, Lemma~\ref{lem:cont} 
and the fact that (\ref{Qa}) is well-defined. 
\end{proof}

\subsection{Proof of Proposition~\ref{prop:rest}}
\begin{proof}
Let $\varphi(m)$
be the Euler function given by 
the order of 
$\mathrm{Gal}(\mathbb{Q}(\xi_m)/\mathbb{Q})=\left(\mathbb{Z}/m\mathbb{Z}\right)^{\ast}$. 
We write $m=m_i \cdot m_i'$ for $1\le i\le l$. 
Since $\Theta_{Q, m}(q)$ is invariant under 
$\xi_m \mapsto g(\xi_m)$ for any element 
$g\in \mathrm{Gal}(\mathbb{Q}(\xi_m)/\mathbb{Q})$, 
we have 
\begin{align*}
&\varphi(m)
\Theta_{Q, m}(q) \\
&=\sum_{(k_1, \cdots, k_n)\in \mathbb{Z}^n}
q^{Q(k_1, \cdots, k_n)}\sum_{g \in \mathrm{Gal}(\mathbb{Q}(\xi_m)/\mathbb{Q})}
g(\xi^{k_1}_m) \\
&= \sum_{i=1}^{l} \sum_{(k_1, \cdots, k_n) \in S_i}
q^{Q(k_1, \cdots, k_n)}\sum_{g \in \mathrm{Gal}(\mathbb{Q}(\xi_m)/\mathbb{Q})}
g(\xi^{k_1}_m) \\
&=\sum_{i=1}^{l}
\sum_{\begin{subarray}{c}
(k_1, \cdots, k_n) \in \mathbb{Z}^n \\
\mathrm{g.c.d.}(k_1, m_i')=1
\end{subarray}}
q^{Q(m_i k_1, k_2, \cdots, k_n)}
\sum_{g \in \mathrm{Gal}(\mathbb{Q}(\xi_m)/\mathbb{Q})}
g(\xi_{m_i'}^{k_1}) \\
&=\sum_{i=1}^{l}
\sum_{\begin{subarray}{c}
(k_1, \cdots, k_n) \in \mathbb{Z}^n \\
\mathrm{g.c.d.}(k_1, m_i')=1
\end{subarray}}
q^{Q(m_i k_1, k_2, \cdots, k_n)}
[\mathbb{Q}(\xi_m) : \mathbb{Q}(\xi_{m_i'})]
\sum_{g \in \mathrm{Gal}(\mathbb{Q}(\xi_{m_i'})/\mathbb{Q})}
g(\xi_{m_i'}^{k_1}). 
\end{align*}
Now the value
\begin{align*}
A_i \cneq [\mathbb{Q}(\xi_m) : \mathbb{Q}(\xi_{m_i'})]
\sum_{g \in \mathrm{Gal}(\mathbb{Q}(\xi_{m_i'})/\mathbb{Q})}
g(\xi_{m_i'}^{k_1})
\end{align*}
is an integer and independent of $k_1 \in \mathbb{Z}$
with $\mathrm{g.c.d.}(k_1, m_i')=1$. 
By setting $Q_i(k_1, \cdots, k_n)=Q(m_i k_1, k_2, \cdots, k_n)$, 
we obtain
\begin{align*}
\Theta_{Q, m}(q)=\frac{1}{\varphi(m)}\sum_{i=1}^{l} A_i 
\sum_{\begin{subarray}{c}
(k_1, \cdots, k_n) \in \mathbb{Z}^n \\
\mathrm{g.c.d.}(k_1, m_i')=1
\end{subarray}}
q^{Q_i(k_1, \cdots, k_n)}. 
\end{align*}
Therefore the result follows from Lemma~\ref{lem:both}. 
\end{proof}

\section{Appendix: Combinatorics on Quot schemes of points on $A_{n-1}$}
\label{sec:Append}
In this appendix,
we describe 
$\chi(\Quot^{m}(\oO_{A_{n-1}}(jD)))$
in terms of certain combinatorial data on Young diagrams. 
In what follows, we regard a Young diagram as a subset
in $\mathbb{Z}_{\ge 0}^{2}$ in the usual way, say:
\begin{align*}
Y=
\begin{array}{l}
\square\!\square\!
\\[-0.21cm]
\square\!\square\!\square
\end{array}
\ \Leftrightarrow \ 
\{(0, 0), (1, 0), (2, 0), (1, 0), (1, 1) \}. 
\end{align*}
Recall that there is a one to one correspondence 
between the set of ideals $I \subset \mathbb{C}[x, y]$
generated by monomials
and that of Young diagrams, by assigning $I$ with the Young diagram $Y_I$: 
\begin{align}\label{YI}
Y_I \cneq \{(a, b) \in \mathbb{Z}_{\ge 0}^2 :
x^a y^b \notin I\}. 
\end{align}
For a Young diagram $Y$, we introduce the following notation: 
\begin{align*}
&Y^{\rightarrow} \cneq \{Y+(1, 0)\} 
\cup \{\{0\} \times \mathbb{Z}_{\ge 0}\} \\
&Y^{\nearrow} \cneq \{Y+(1, 1)\} \cup \{\mathbb{Z}_{\ge 0} \times \{0\}\}
\cup \{\{0\} \times \mathbb{Z}_{\ge 0}\}. 
\end{align*}
Note that $Y^{\rightarrow}$ and $Y^{\nearrow}$
are Young diagrams with 
infinite number of 
blocks. See the following picture: 
\begin{align*}
Y=
\begin{array}{l}
\square\!\square\!
\\[-0.21cm]
\square\!\square\!\square
\end{array}
\ \Rightarrow \ 
Y^{\rightarrow}=
\begin{array}{l}
\vdots
\\[-0.10cm]
\square\!
\\[-0.21cm]
\square\!
\\[-0.21cm]
\square\!
\\[-0.21cm]
\square\!\square\!\square\!
\\[-0.21cm]
\square\!\square\!\square\!\square\!
\end{array} \ 
Y^{\nearrow}=
\begin{array}{l}
\vdots
\\[-0.10cm]
\square\!
\\[-0.21cm]
\square\!
\\[-0.21cm]
\square\!\square\!\square\!
\\[-0.21cm]
\square\!\square\!\square\!\square\!
\\[-0.21cm]
\square\!\square\!\square\!\square\!\square\!\square\!\cdots
\end{array}
\end{align*}
\begin{lem}\label{lem:Young}
For $0\le j\le n-1$, the 
number
 $\chi(\Quot^m(\oO_{A_{n-1}}(-jD)))$ coincides with the number of 
$n$-tuples of Young diagrams 
$(Y_0, Y_1, \cdots, Y_{n-1})$ satisfying 
\begin{align}\label{Young}
Y_{n-1} \subset \cdots \subset Y_j \subset 
Y_{j-1}^{\rightarrow} \subset \cdots \subset Y_0^{\rightarrow} \subset 
Y_{n-1}^{\nearrow}, 
 \quad 
\sum_{i=0}^{n-1} \lvert Y_i \rvert =m. 
\end{align}
\end{lem}
\begin{proof}
Giving a point in $\Quot^m(\oO_{A_{n-1}}(-jD))$
is equivalent to giving an ideal 
$I \subset \oO_{A_{n-1}}$ 
such that $I \subset \oO_{A_{n-1}}(-jD)$
and $\oO_{A_{n-1}}(-jD)/I$ is a $m$-dimensional $\mathbb{C}$-vector space. 
As a $\mathbb{C}$-vector space, we have the decomposition
\begin{align}\label{I:decompose}
I=\bigoplus_{k=0}^{n-1} I_k \cdot z^k
\end{align}
for sub vector spaces $I_k \subset \mathbb{C}[x, y]$. 
Since $I$ is an ideal in $\oO_{A_{n-1}}$, 
 each $I_k$ is an ideal in $\mathbb{C}[x, y]$. 
Moreover since $I$ must be closed under the multiplication by $z$, 
we have 
\begin{align}\label{I:seq}
xy I_{n-1} \subset I_0 \subset I_1 \subset \cdots \subset I_{n-1}. 
\end{align}
Since $\oO_{A_{n-1}}(-jD)=(x, z^j)$, the 
condition $I \subset \oO_{A_{n-1}}(-jD)$
is equivalent to 
$I_k \subset (x)$ for $0\le k \le j-1$. 
Hence for $0\le k\le j-1$, we have 
$I_k=I_k' \cdot (x)$ for some ideal 
$I_k' \subset \mathbb{C}[x, y]$. 
We obtain the sequence of ideals in 
$\mathbb{C}[x, y]$:
\begin{align}\label{seq:ideal}
I_0', \cdots, I_j', I_{j+1}, \cdots I_{n-1}. 
\end{align}
The condition 
that $\oO_{A_{n-1}}(-jD)/I$ 
is $m$-dimensional is equivalent to 
\begin{align}\label{sum=m}
\sum_{k=0}^{j-1}
\dim \mathbb{C}[x, y]/I_k' + 
\sum_{k=j}^{n-1}\dim \mathbb{C}[x, y]/I_k =m. 
\end{align}
Conversely suppose that we have
a sequence of ideals (\ref{seq:ideal})
in $\mathbb{C}[x, y]$ satisfying 
(\ref{sum=m}) and (\ref{I:seq})
for $I_k=I_k' \cdot (x)$ with $0\le k\le j-1$.
Then 
we obtain an ideal
$I \subset \oO_{A_{n-1}}$ by 
setting (\ref{I:decompose}),
which gives a point in $\Quot^m(\oO_{A_{n-1}}(-jD))$. 
Note that $T=(\mathbb{C}^{\ast})^{\times 2}$ acts 
on $A_{n-1}$ via $(t_1, t_2) \cdot (x, y, z)=(t_1^n x, t_2 ^n y, t_1 t_2 z)$, 
and the ideal (\ref{I:decompose}) is $T$-fixed if and only if 
the corresponding 
ideals in (\ref{seq:ideal})
are generated by monomials.
Therefore the $T$-fixed locus of 
$\Quot^m(\oO_{A_{n-1}}(-jD))$
is identified with 
 the set of $n$-tuples of Young diagrams $(Y_0, \cdots, Y_n)$
satisfying (\ref{Young}),
by assigning a sequence (\ref{seq:ideal})
with 
\begin{align*}
(Y_0, \cdots, Y_n)=
(Y_{I_0'}, \cdots, Y_{I_{j-1}'}, 
Y_{I_j}, \cdots, Y_{I_{n-1}})
\end{align*} 
as in (\ref{YI}). By the $T$-localization, 
we obtain the desired result. 
\end{proof}

\begin{rmk}
The number $\chi(\Quot^m(\oO_{A_{n-1}}(-jD)))$ in 
Lemma~\ref{lem:Young} and the coefficients in 
the LHS of 
(\ref{Quotm}) are related by 
\begin{align}\label{relate:chi}
\chi(\Quot^m(\oO_{A_{n-1}}(-jD)))=
\chi(\Quot^m(\oO_{A_{n-1}}((n-j)D)))
\end{align}
as $\oO_{A_{n-1}}(nD) \cong \oO_{A_{n-1}}$. 
\end{rmk}
We compare the formula (\ref{Quotm})
with the numbers of $n$-tuples of Young 
diagrams in Lemma~\ref{lem:Young}
in examples: 
\begin{exam}
(i) If $n=2$ and $j=0$, then the formula (\ref{Quotm}) 
implies 
\begin{align*}
\sum_{m\ge 0} \chi(\Hilb^m(A_1))q^m
= 1+ q+3q^2 + 5q^3 + 9q^4 + 14q^5 + \cdots.
\end{align*}
For instance, $\chi(\Hilb^5(A_1))$
corresponds to the following 14 pairs of Young diagrams $(Y_0, Y_1)$:  
\begin{align*}
&\left(
\begin{array}{l}
\square\!
\\[-0.21cm]
\square\!
\\[-0.21cm]
\square\!
\\[-0.21cm]
\square\!
\\[-0.21cm]
\square\!
\end{array}, \emptyset
\right), \ 
\left(
\begin{array}{l}
\square\!
\\[-0.21cm]
\square\!
\\[-0.21cm]
\square\!
\\[-0.21cm]
\square\!\square\!
\end{array}, \emptyset
\right), \ 
\left(
\begin{array}{l}
\square\!
\\[-0.21cm]
\square\!
\\[-0.21cm]
\square\!\square\!\square\!
\end{array}, \emptyset
\right), \ 
\left(
\begin{array}{l}
\square\!
\\[-0.21cm]
\square\!\square\!\square\!\square\!
\end{array}, \emptyset
\right), \\
&\left(
\begin{array}{l}
\square\!\square\!\square\!\square\!\square\!
\end{array}, \emptyset
\right), \
\left(
\begin{array}{l}
\square\!
\\[-0.21cm]
\square\!
\\[-0.21cm]
\square\!
\\[-0.21cm]
\square\!
\end{array}, 
\begin{array}{l}
\empty
\\[-0.21cm]
\empty
\\[-0.21cm]
\empty
\\[-0.21cm]
\square\!
\end{array}
\right), \ 
\left(
\begin{array}{l}
\square\!
\\[-0.21cm]
\square\!
\\[-0.21cm]
\square\!\square\!
\end{array}, 
\begin{array}{l}
\empty
\\[-0.21cm]
\empty
\\[-0.21cm]
\square\!
\end{array}
\right), \ 
\left(
\begin{array}{l}
\square\!\square\!
\\[-0.21cm]
\square\!\square\!
\end{array}, 
\begin{array}{l}
\empty
\\[-0.21cm]
\square\!
\end{array}
\right), \\
&\left(
\begin{array}{l}
\square\!
\\[-0.21cm]
\square\!\square\!\square\!
\end{array}, 
\begin{array}{l}
\empty
\\[-0.21cm]
\square\!
\end{array}
\right), \
\left(
\begin{array}{l}
\square\!\square\!\square\!\square\!
\end{array}, 
\begin{array}{l}
\square\!
\end{array}
\right), \
\left(
\begin{array}{l}
\square\!
\\[-0.21cm]
\square\!
\\[-0.21cm]
\square\!
\end{array}, 
\begin{array}{l}
\empty
\\[-0.21cm]
\square\!
\\[-0.21cm]
\square\!
\end{array}
\right), \
\left(
\begin{array}{l}
\square\!
\\[-0.21cm]
\square\!\square\!
\end{array}, 
\begin{array}{l}
\square\!
\\[-0.21cm]
\square\!
\end{array}
\right), \\
&\left(
\begin{array}{l}
\square\!
\\[-0.21cm]
\square\!\square\!
\end{array}, 
\begin{array}{l}
\empty
\\[-0.21cm]
\square\!\square\!
\end{array}
\right), \
\left(
\begin{array}{l}
\square\!\square\!\square\!
\end{array}, 
\begin{array}{l}
\square\!\square\!
\end{array}
\right).
\end{align*}

(ii) If $n=2$ and $j=1$, then (\ref{Quotm})
and (\ref{relate:chi})
yield
\begin{align*}
\sum_{m\ge 0} \chi(\Quot^m(\oO_{A_1}(-D)))q^m
= 1+ 2q+3q^2 + 6q^3 + 10q^4 + 16q^5 + \cdots.
\end{align*}
Similarly to (i), $\chi(\Quot^5(\oO_{A_1}(-D)))$
corresponds to the following 16 pairs of Young diagrams $(Y_0, Y_1)$:  
\begin{align*}
&\left(
\begin{array}{l}
\square\!\square\!\square\!\square\!\square\!
\end{array}, \emptyset
\right), \ 
\left(
\begin{array}{l}
\square\!\square\!\square\!\square\!
\end{array}, 
\begin{array}{l}
\square\!
\end{array}
\right), \ 
\left(
\begin{array}{l}
\square\!
\\[-0.21cm]
\square\!\square\!\square\!
\end{array}, 
\begin{array}{l}
\empty
\\[-0.21cm]
\square\!
\end{array}
\right), \ 
\left(
\begin{array}{l}
\square\!\square\!\square\!
\end{array}, 
\begin{array}{l}
\square\!\square\!
\end{array}
\right), \\
&\left(
\begin{array}{l}
\square\!
\\[-0.21cm]
\square\!\square\!
\end{array}, 
\begin{array}{l}
\empty
\\[-0.21cm]
\square\!\square\!
\end{array}
\right), \ 
\left(
\begin{array}{l}
\square\!
\\[-0.21cm]
\square\!
\\[-0.21cm]
\square\!
\end{array}, 
\begin{array}{l}
\empty
\\[-0.21cm]
\square\!
\\[-0.21cm]
\square\!
\end{array}
\right), \
\left(
\begin{array}{l}
\square\!
\\[-0.21cm]
\square\!\square\!
\end{array}, 
\begin{array}{l}
\square\!
\\[-0.21cm]
\square\!
\end{array}
\right), \
\left(
\begin{array}{l}
\empty
\\[-0.21cm]
\square\!\square\!\square\!
\end{array}, 
\begin{array}{l}
\square\!
\\[-0.21cm]
\square\!
\end{array}
\right), \\
&\left(
\begin{array}{l}
\empty 
\\[-0.21cm]
\empty
\\[-0.21cm]
\square\!\square\!
\end{array}, 
\begin{array}{l}
\square\!
\\[-0.21cm]
\square\!
\\[-0.21cm]
\square\!
\end{array}
\right), \ 
\left(
\begin{array}{l}
\empty
\\[-0.21cm]
\square\!\square\!
\end{array}, 
\begin{array}{l}
\square\!
\\[-0.21cm]
\square\!\square\!
\end{array}
\right), \
\left(
\begin{array}{l}
\square\!\square\!
\end{array}, 
\begin{array}{l}
\square\!\square\!\square\!
\end{array}
\right), \
\left(
\begin{array}{l}
\empty
\\[-0.21cm]
\square\! 
\\[-0.21cm]
\square\!
\end{array}, 
\begin{array}{l}
\square
\\[-0.21cm]
\square\! 
\\[-0.21cm]
\square\!
\end{array}
\right), \\ 
&\left(
\begin{array}{l}
\square\! 
\\[-0.21cm]
\square\!
\end{array}, 
\begin{array}{l}
\square\! 
\\[-0.21cm]
\square\!\square\!
\end{array}
\right), \
\left(
\begin{array}{l}
\empty
\\[-0.21cm]
\empty
\\[-0.21cm]
\empty
\\[-0.21cm]
\square\!
\end{array}, 
\begin{array}{l}
\square\!
\\[-0.21cm]
\square\!
\\[-0.21cm]
\square\!
\\[-0.21cm]
\square\!
\end{array}
\right), \
\left(
\begin{array}{l}
\empty
\\[-0.21cm]
\empty
\\[-0.21cm]
\square\!
\end{array}, 
\begin{array}{l}
\square\!
\\[-0.21cm]
\square\!
\\[-0.21cm]
\square\!\square\!
\end{array}
\right), \
\left(
\emptyset, 
\begin{array}{l}
\square\!
\\[-0.21cm]
\square\!
\\[-0.21cm]
\square\!
\\[-0.21cm]
\square\!
\\[-0.21cm]
\square\!
\end{array}
\right).
\end{align*}
\end{exam}

\providecommand{\bysame}{\leavevmode\hbox to3em{\hrulefill}\thinspace}
\providecommand{\MR}{\relax\ifhmode\unskip\space\fi MR }
\providecommand{\MRhref}[2]{%
  \href{http://www.ams.org/mathscinet-getitem?mr=#1}{#2}
}
\providecommand{\href}[2]{#2}

Kavli Institute for the Physics and 
Mathematics of the Universe, University of Tokyo,
5-1-5 Kashiwanoha, Kashiwa, 277-8583, Japan.

\textit{E-mail address}: yukinobu.toda@ipmu.jp


\begin{thebibliography}{BvdGHZ07}

\bibitem[BF08]{BBr}
K.~Behrend and B.~Fantechi, \emph{Symmetric obstruction theories and {H}ilbert
  schemes of points on threefolds}, Algebra Number Theory \textbf{2} (2008),
  313--345.

\bibitem[Bri02]{Br1}
T.~Bridgeland, \emph{Flops and derived categories}, Invent. Math \textbf{147}
  (2002), 613--632.

\bibitem[Bri11]{BrH}
\bysame, \emph{Hall algebras and curve-counting invariants},
  J.~Amer.~Math.~Soc.~ \textbf{24} (2011), 969--998.

\bibitem[BvdGHZ07]{Z123}
J.~H. Bruinier, G.~van~der Geer, G.~Harder, and D.~Zagier, \emph{The 1-2-3 of
  {M}odular {F}orms}, Letures at a Summer School in Nordfjordeid, Norway,
  Springer, 2007.

\bibitem[Cal]{Cala}
J.~Calabrese, \emph{Donaldson-{T}homas invariants on {F}lops}, preprint,
  arXiv:1111.1670.

\bibitem[dB04]{MVB}
M.~Van den Bergh, \emph{Three dimensional flops and noncommutative rings}, Duke
  Math.~J.~ \textbf{122} (2004), 423--455.

\bibitem[G\"90]{Got}
L.~G\"ottsche, \emph{The {B}etti numbers of the {H}ilbert scheme of points on a
  smooth projective surface}, Math.~Ann.~ \textbf{286} (1990), 193--207.

\bibitem[GS]{GS3}
A.~Gholampour and A.~Sheshmani, \emph{Donaldson-{T}homas {I}nvariants of
  2-{D}imensional sheaves inside threefolds and modular forms}, preprint,
  arXiv:1309.0050.

\bibitem[JS12]{JS}
D.~Joyce and Y.~Song, \emph{A theory of generalized {D}onaldson-{T}homas
  invariants}, Mem.~Amer.~Math.~Soc.~ \textbf{217} (2012).

\bibitem[KS]{K-S}
M.~Kontsevich and Y.~Soibelman, \emph{Stability structures, motivic
  {D}onaldson-{T}homas invariants and cluster transformations}, preprint,
  arXiv:0811.2435.

\bibitem[Li06]{Li}
J.~Li, \emph{Zero dimensional {D}onaldson-{T}homas invariants of threefolds},
  Geom.~Topol.~ \textbf{10} (2006), 2117--2171.

\bibitem[LP09]{LP}
M.~Levine and R.~Pandharipande, \emph{Algebraic cobordism revisited},
  Invent.~Math.~ \textbf{176} (2009), 63--130.

\bibitem[Mau]{MH}
D.~Maulik, \emph{Stable pairs and the {HOMFLY} polynomial}, preprint,
  arXiv:1210.6323.

\bibitem[MNOP06]{MNOP}
D.~Maulik, N.~Nekrasov, A.~Okounkov, and R.~Pandharipande,
  \emph{Gromov-{W}itten theory and {D}onaldson-{T}homas theory. {I}},
  Compositio.~Math \textbf{142} (2006), 1263--1285.

\bibitem[OS12]{SV}
A.~Oblomkov and V.~Shende, \emph{The {H}ilbert scheme of a plane curve
  singularity and the {HOMFLY} polynomial of its link}, Duke Math.~J.~
  \textbf{161} (2012), 1277--1303.

\bibitem[PT10]{PT3}
R.~Pandharipande and R.~P. Thomas, \emph{Stable pairs and {BPS} invariants},
  J.~Amer.~Math.~Soc.~ \textbf{23} (2010), 267--297.

\bibitem[Rei]{Rei}
M.~Reid, \emph{Minimal models of canonical 3-folds}, Algebraic Varieties and
  Analytic Varieties (S.~Iitaka, ed), Adv. Stud. Pure Math, Kinokuniya, Tokyo,
  and North-Holland, Amsterdam \textbf{1}, 131--180.

\bibitem[Sta99]{Stan}
R.~Stanley, \emph{Enumerative combinatorics}, Cambridge University Press, 1999.

\bibitem[Tho00]{Thom}
R.~P. Thomas, \emph{A holomorphic {C}asson invariant for {C}alabi-{Y}au 3-folds
  and bundles on ${K3}$-fibrations}, J.~Differential.~Geom \textbf{54} (2000),
  367--438.

\bibitem[Tod]{TodS}
Y.~Toda, \emph{Flops and {S}-duality conjecture}, preprint, arXiv:1311.7476.

\bibitem[Tod10]{Tcurve1}
\bysame, \emph{Curve counting theories via stable objects~{I}: {DT/PT}
  correspondence}, J.~Amer.~Math.~Soc.~ \textbf{23} (2010), 1119--1157.

\bibitem[Tod13]{Tcurve2}
\bysame, \emph{Curve counting theories via stable objects~{II}. {DT}/nc{DT}
  flop formula}, J.~Reine Angew.~Math.~ \textbf{675} (2013), 1--51.

\bibitem[VW94]{VW}
C.~Vafa and E.~Witten, \emph{A {S}trong {C}oupling {T}est of {S}-{D}uality},
  Nucl.~Phys.~B \textbf{431} (1994).

\end{thebibliography}
\end{document}